\documentclass[11pt]{amsart}
\usepackage{amsmath,amsthm,amsfonts,latexsym,amssymb}
\usepackage{graphicx}
\usepackage[arrow, matrix, curve]{xy}
\usepackage{tikz-cd}
\usepackage{float}
\usepackage{hyperref}
\usepackage{xcolor}

\newcommand{\C}{\mathbb{C}}
\newcommand{\R}{\mathbb{R}}

\newtheorem{theorem}{Theorem}[section]
\newtheorem{lemma}[theorem]{Lemma}
\newtheorem{proposition}[theorem]{Proposition}

\theoremstyle{definition}

\newtheorem{example}[theorem]{Example}
\newtheorem{remark}[theorem]{Remark}

\newtheorem{definition}[theorem]{Definition}

\makeatletter
\@namedef{subjclassname@2020}{%
	\textup{2020} Mathematics Subject Classification}
\makeatother

\title{Classification of Legendrian doubles and suspensions}
\author{Yasemin Yildirim}
\address{Department of Mathematics, Uppsala University, Box 480, 751 06, Uppsala, Sweden} 
\address{Department of Mathematics, TED University, Ankara, Turkiye}
\email{yasemin.yildirim@math.uu.se, yyildirim@tedu.edu.tr}

\date{\today}
\begin{document}
\begin{abstract}
We define a construction of Legendrians inside contact manifolds that arise by doubling an exact Lagrangian filling in the page of an open book decomposition. This can be seen as a generalization of a previous construction by Courte and Ekholm  to arbitrary open books. These Legendrians, called Legendrian doubles, are shown to admit regular flexible exact Lagrangian fillings, and they are thus classified up to Legendrian isotopy by classical data. Finally, we show that the Legendrian suspension construction, as defined by Arikan and the author in  previous work,—this is a Legendrian contained inside a page of an open book that is obtained by using Seidel’s suspension of Lefschetz fibrations— is a Legendrian double.
\end{abstract}
\maketitle
\section{Introduction}

The Legendrian submanifolds of a contact manifold constitute a very rich set of objects. The relation of Legendrian embeddings up to Legendrian isotopy, i.e.~the classification of Legendrians up to an ambient contact isotopy, is far from well understood. However, there are two very particular classes of Legendrian submanifolds for which more can be said, for the reason that their classical contact invariants are enough to determine their Legendrian isotopy class; namely, the Legendrians that are loose in the sense of Murphy \cite{Murphy}, and the Legendrians that admit a flexible exact Lagrangian filling, a notion which only exists in contact manifolds of dimensions at least $5$ which was introduced by \cite{EliashbergGanatra}. The class of loose Legendrians satisfy an $h$-principle by Murphy's work, while the second ones admit exact Lagrangian fillings that satisfy an $h$-principle \cite{EliashbergGanatra}. In both cases, once it has been verified that the Legendrian is in the class, one can use softer techniques from contact topology in order to determine its Legendrian isotopy class. Note that the loose Legendrians are known not to admit any exact Lagrangian fillings inside any symplectic filling of the contact manifold. On the other hand, the class of flexible fillable Legendrians should be considered as the most simple ones of those that admit exact Lagrangian fillings.

Here we provide a new construction of closed Legendrian submanifolds of a contact manifold $(Y,\xi)$ that we call Legendrian doubles. The construction starts with the choice of a compatible contact open book decomposition of the contact manifold $(Y,\xi)$. Recall that, by the work of Giroux \cite{Giroux2002}, there are always plenty of such open book decompositions. We then fix an exact Lagrangian submanifold $L \subset ({P},\eta)$ of a page of the open book that is cylindrical near the boundary (recall that the boundary of the page is equal to the binding of the open book). This Lagrangian $L$ can be doubled by adjoining a copy of $L$ in the opposite page, which makes an 180 degree angle with the page that we started with. The closed and smooth submanifold obtained is diffeomorphic to $L \cup_{\partial L} L$, and becomes a Legendrian submanifold $L_{Leg}^d$ after a suitable perturbation away from the binding, see Lemma \ref{lemma_def_ob}. Note that this perturbation is not necessary in the case when $L$ is strongly exact (for definition see Section \ref{prelim}). We call the latter the Legendrian double of $L$. See  Definition \ref{definition:Leg_double} and Section \ref{doubling process} for more details.

Our main result concerning the Legendrian doubles is that they admit flexible regular exact Lagrangian fillings. In other words, they are Legendrians that admit exact Lagrangian fillings of the simplest types. 
\begin{theorem}\label{theorem:flexible_filling}
    If $L$ is an exact Lagrangian filling of dimension at least $2$ of a Legendrian submanifold $\Lambda$ in the   page ${P}$ of an open book decomposition of the contact manifold $(Y,\xi)$, then the Legendrian submanifold $L_{Leg}^d$ obtained by the double of  this filling bounds a regular flexible exact Lagrangian filling in the symplectization. Moreover, the Legendrian isotopy class of the double only depends on the formal Lagrangian isotopy class of $L$.
\end{theorem}

 The uniqueness of the Legendrian isotopy class more or less follows from the work of \cite{CE}. Also, see \cite{roy2022constructions} for the case when the Lagrangian filling is a disk. 
 
 In fact, we prove a possibly even stronger statement. Namely, the produced by Theorem \ref{thm:flex_trivial_fib} filling is a  \emph{decomposable Lagrangian filling}.

In high dimensions, \emph{decomposability}, \emph{regularity}, and \emph{flexibility} are distinct but related notions. A Lagrangian filling is called \emph{decomposable} if it can be constructed from the standard filling of unlinked standard Legendrian spheres by a finite sequence of elementary Lagrangian cobordisms, i.e.~ Legendrian isotopies and standard Lagrangian handle attachments. This notion was first introduced by \cite{chantraine2010concordance} in 4-dimensonal symplectic cobordisms, but it can be generalized to higher dimensions by using the notion of standard Lagrangian handle, \cite{Dimitroglou:Ambient}. The  notion of regular Lagrangian cobordisms is more recent and it goes back to \cite{EliashbergGanatra}.  Roughly speaking, a cobordism is \emph{regular} if its complementary Liouville cobordism is Weinstein. When the complementary cobordism is moreover flexible, which makes sense in symplectic cobordisms of dimension at least 6, we call the Lagrangian \emph{flexible}. Such flexible Lagrangian cobordisms satisfy an $h$-principle, as shown in the same paper. There is a stronger relationship between the notions of decomposable Lagrangian cobordisms and regular Lagrangian cobordisms as shown in \cite{conway2021symplectic} in low dimensions and \cite[Proposition 3.2]{DRG} in high dimensions. Note that, decomposable $n+1$-dimensional Lagrangian cobordism for which all Lagrangian standard handles are of index $<n$ is flexible by \cite{EliashbergGanatra}. Thus, decomposability is primarily a statement about the construction of the filling, whereas regularity concerns its compatibility with the ambient Liouville structure. By contrast, \emph{flexibility} is a geometric property associated with the underlying Weinstein structure of the complementary cobordism, characterized by the presence of loose Legendrian attaching spheres for its critical handles. Flexible Weinstein structures therefore satisfy an \(h\)-principle, making their associated Lagrangian geometry considerably more tractable. It is an open problem whether all regular cobordisms are decomposable, but the flexible regular ones can be shown to be decomposable by alluding to the $h$-principle from \cite{EliashbergGanatra}. In the same paper, the authors conjectured that every exact Lagrangian filling inside a Weinstein domain is regular.

In the case when the open book decomposition is induced by trivial Lefschetz fibration $\operatorname{proj}_\C:\hat{P}\times \C \to \C$, then doubling a strongly exact filling $L\subset (P,\eta)$ produces the ideal boundary at infinity of $\hat{L}\times \R \subset \hat{P}\times \C$. Then we have the following result.
\begin{theorem}\label{thm:flex_trivial_fib}
  The product Lagrangian $\hat{L}\times \R \subset \hat{P} \times \C $ is a complete flexible Lagrangian filling of $L$ with boundary given by the Legendrian double $L^d_{Leg}$ created using the trivial open book with  Weinstein page $(P,\eta)$.  
\end{theorem}
    
It is worth emphasizing that this result is more or less immediate when $L$ is regular. However, we make no assumptions on the regularity of $L$.

In the special case when $L$ is a Lagrangian disk, the Legendrian double is the standard Legendrian embedding. As we show in Section \ref{section_symp_ribbon}, it turns out that this Legendrian double is useful for giving an explicit description of the page in the open book decomposition obtained by stabilization. Recall that a Lagrangian embedding of a disk in the page can be used to produce a so-called stabilization of the open book, which is a new open book compatible with the same contact manifold, where the new page is symplectomorphic to the result of a Weinstein handle attachment along the boundary of the Lagrangian disk. Note that the stabilized page contains natural Legendrian sphere. In Theorem \ref{prop_stab_ob}, we give a precise relation between the page in the stabilization and the old page union a symplectic ribbon of the Legendrian double of the disk along which the stabilization is performed. 

We then relate Legendrian doubling to a different construction that produces a Legendrian contained in the page of a compatible open book. In previous joint work by the author and Arıkan \cite{AY}, the Legendrian suspension was constructed in the following manner. Again, we assume that we are given an exact Lagrangian filling $L \subset ({P},\eta)$ contained inside the page of a compatible open book on $(Y,\xi)$ with binding $B=\partial {P}$. Furthermore, we assume that the boundary of $L$ is contained inside a page $P_B$ of a compatible open book of the binding $B$ that arises from a symplectic Lefschetz fibration $\pi \colon \hat{P} \to \C$ of the completion $\hat{P}$ of the page ${P}$. We consider the completion $\hat{L} \subset \hat{P}$ of an exact Lagrangian filling $L \subset P$ which outside of a compact subset projects to a line $(R-\epsilon,+\infty)$ under $\pi$, and where the intersection of $\hat{L}$ with $\pi^{-1}(R)$ is the Legendrian boundary $\Lambda = \partial L$. Seidel's suspension construction produces a closed Lagrangian $L^\sigma \subset \hat{P}^\sigma$ contained inside the fiber $\hat{P}^\sigma=(\pi^\sigma)^{-1}(R)$ of the Lefschetz fibration $\pi^\sigma=\pi+z^2 \colon \hat{P} \times \C_z \to \C$ given as a stabilization of the trivial Lefschetz fibration $\operatorname{proj}_\C:\hat{P}\times \C\to \C$. The Weinstein manifold $\hat{P}^\sigma$ can be presented in the following different manners:
\begin{itemize}
    \item the double cover $\hat{P}^\sigma \to \hat{P}$ that is branched over the regular fiber $\hat{P}_B=\pi^{-1}(R)$;
   
    \item $\hat{P}^\sigma$ is obtained from $\hat{P}$ by attaching top Weinstein handles on all vanishing cycles induced by $\pi$, and extending the Lefschetz fibration with additional critical point on each added handle.
\end{itemize}

The Lagrangian $L^\sigma$ is by construction the pre-image of $L$ under the above branched cover. In other words, it is diffeomorphic to $L^\sigma=L \cup_{\partial L} L$, i.e.~the double of $L$. See Definition \ref{definition:Leg_suspension}.

Since $P^\sigma$ also can be considered as the page of the open book of $(Y,\xi)$ after stabilizing the page $P$ in the same manner as the above Lefschetz fibration was stabilized, it follows that $L^\sigma$ can be naturally considered as a Legendrian inside the page of this stabilized open book. This is the key feature of the stabilization construction: it produces Legendrians that are naturally embedded in the page of an open book. 

We finally relate the concepts of Legendrian doubles and suspension by proving the following, see Section \ref{suspension} for more.

\begin{theorem}\label{theorem:leg_isotopy}
The Legendrian suspension $L^\sigma\subset \partial_\infty (\hat{P}\times \C)$ of an exact Lagrangian filling $L\subset P$ whose completion $\hat{L} \subset \hat{P}$ projects onto $(R-\epsilon,+\infty)$ under the Lefschetz fibration $\pi \colon \hat{P} \to \C$ is Legendrian isotopic to the Legendrian double $L_{Leg}^d$ induced by the trivial open book on $\partial_\infty (\hat{P}\times \C)$ arising from the trivial Lefschetz fibration.
\end{theorem}
\section{Preliminaries} \label{prelim}

\subsection{Basics of contact and symplectic topology}
We begin with a notion of a contact manifold. A \textit{contact manifold} $(Y,\xi)$ is a $2n+1$-dimensional smooth manifold equipped with a maximally non-integrable hyperplane distribution  $\xi$ of $TY$. The distribution $\xi$ is called a \textit{contact structure} on $Y$ and locally is given by $\xi$ that satisfies $\alpha\wedge (d\alpha)^{n} \neq 0$ for a  one-form $\alpha$ on $Y$ that defines $\xi$ as its kernel.  Such a differential one-form $\alpha$ is called a \textit{contact form} on $Y$. A submanifold $\Lambda$ of the contact manifold $ (Y,\xi)$ is called \textit{isotropic} if $\Lambda$ is everywhere tangent to $\xi$ and is called \textit{Legendrian} if it is of maximal dimension, i.e $n$.  

Given a contact form $\alpha$ defining a contact structure \(\xi\) on a contact manifold \(Y\), the associated \emph{Reeb vector field} \(\mathcal{R}_\alpha\) is the unique vector field on \(Y\) satisfying $$
\alpha(\mathcal{R}_\alpha)=1,
\qquad
\iota_{\mathcal{R}_\alpha}d\alpha=0.
$$
Equivalently, $$
d\alpha(\mathcal{R}_\alpha,\cdot)=0.
$$
Building upon this foundation, we now introduce related concepts essential to the discussion of exact Lagrangian submanifolds.

An \textit{exact symplectic manifold} $(X,\omega)$ is an even dimensional smooth manifold together with a closed,  non-degenerate and exact two-form $\omega=d\eta$ with a primitive one-form $\eta$. A half dimensional submanifold $L$ of the exact symplectic manifold $ X$  is said to be \textit{Lagrangian} if the restriction of the symplectic two-form to the submanifold $L$ vanishes, and \emph{exact} if the pull-back of $\eta$ to $L$ moreover is exact.  We call a Lagrangian submanifold \emph{strongly exact} if the pull-back of the primitive to $L$ vanishes, i.e. $\eta|_{TL}=0.$

Further extending these ideas, the concept of a symplectization of a contact manifold becomes relevant, particularly when considering Lagrangian cobordisms. 

The \textit{symplectization} of the contact manifold $(Y, \xi)$ is the exact symplectic manifold $(\mathbb{R}\times Y, d(e^{t}\alpha))$ where $t$ is the $\mathbb{R}$-coordinate and $\alpha$ is a contact form.

Given two Legendrian submanifolds $\Lambda_-$ and $\Lambda_+$ of $(Y,\xi)$, an \textit{exact Lagrangian cobordism from} $\Lambda_-$ to $\Lambda_+$ is a properly embedded exact Lagrangian submanifold $\hat{L}$ of the symplectization $(\mathbb{R}\times Y,d(e^{t}\alpha))$ in the following form:
$$ \hat{L}=(-\infty,-T]\times \Lambda_-\cup {L}\cup [T,\infty)\times \Lambda_+$$
for some positive number $T$ where ${L}$ is a subset of $[-T,T]\times Y$ with boundary $\partial {L}=\Lambda_-\sqcup \Lambda_+$. If the negative boundary $\Lambda_-$ is empty, then $\hat{L}$ is called an \textit{exact Lagrangian filling} of $\Lambda_+$.

At this stage, exact Lagrangian fillings provide a natural bridge between Legendrian geometry and symplectic topology. In particular, the ambient exact symplectic manifolds in which such fillings live often arise from a special class of manifolds known as Liouville domains. These domains serve  as the foundational building blocks for many constructions in symplectic and contact topology, including symplectic fillings, open books and Lefschetz fibrations. To proceed, we first review the structure of Liouville domains and their completions.

 A \textit{Liouville domain} is a compact exact symplectic manifold $({P}, \eta)$ with boundary equipped with globally defined Liouville vector field $\chi$, i.e. $\mathcal{L}_\chi \omega=\omega$ for the symplectic two-form $\omega=d\eta$, which points transversally out of the boundary. This implies that the boundary has an induced contact form $\left(\partial{P},\alpha_\eta=\eta|_{T(\partial{P})}\right).$

 A key concept for a Liouville domain $({P}, \eta, \chi)$ is the skeleton of the Liouville flow, which can be formally expressed as 

 $$ \operatorname{Skel}({P}, \eta, \chi)=\bigcap\limits_{t>0} \phi^{-t}({P})$$

\noindent where $\phi^t \colon ({P},e^t\eta) \to ({P},\eta)$ the conformal symplectomorphism induced by the time-$t$ flow of the Liouville vector field, the so-called \textit{Liouville flow}. The skeleton is the attractor of the negative flow generated by $\chi$ which implies that a point in the domain belongs to the skeleton if and only if it does not escape the domain under this flow. Note that ${P}$  deformation retracts onto its skeleton by the flow.

 A Liouville domain ${P}$ can be extended to a Liouville manifold by attaching the symplectization of the boundary $\partial {P}$ to the boundary of ${P}$. The resulting manifold 

 $$ \hat{P}:= {P} \cup_{\partial {P}} \left((0,\infty)\times \partial {P}\right)$$

 \noindent is called the \textit{symplectic completion} of the Liouville domain ${P}$. By a Liouville manifold we will mean an exact symplectic manifold that has been obtained by the completion of a Liouville domain in this manner.

\subsection{Open books and Lefschetz fibrations}
 Symplectic Lefschetz fibration, when it exists, is a very important tool for describing and understanding a Liouville manifold. A Lefschetz fibration on a Liouville domain induces a compatible open book decomposition of its contact boundary. Contact manifolds in general carry compatible open book decompositions, as proven by Giroux \cite{Giroux2002}. They also provide crucial tools for investigating the contact manifold.  The connection between these two phenomena is fundamental in contact topology and symplectic geometry.

 An \textit{abstract contact open book} $({P},\eta,\phi)$  consists of  a $2n$-dimensional Liouville domain $({P},\eta)$   and a symplectomorphism
$$\phi \colon ({P},\eta) \to ({P},\eta)$$
  that is the identity near $\partial {P}$ and  satisfies $\phi^*\eta=\eta+dG$ for some $G \colon {P} \to \R$ which vanishes near the boundary.
  
 An abstract contact open book gives rise to a closed contact manifold whose contact structure is supported by the open book, a notion introduced by Giroux, who also emphasized its importance in \cite{Giroux2002}. By construction, this contact manifold is supported by the same open book in sense of the following definition.
 
 \begin{definition}
    A contact manifold \((Y,\xi)\) equipped with a contact form $\alpha$ is said to be \emph{supported} by an open book with page \((P,\eta)\) and monodromy \(\phi \in \mathrm{Symp}^c(P\setminus \partial P)\) if there is a smooth identification of \(Y\) and the corresponding open book such that the contact form satisfies the following properties:
\begin{itemize}
\item Near the binding: There is a neighborhood $D^2_{3\epsilon} \times \partial P$ of the binding, the contact form $\alpha$ is of the form
\(        e^{f(r)}\alpha_\eta + g(r)\,d\theta,
   \)
   using polar coordinates, where moreover, \begin{itemize}
    \item $f(r)=-\frac{r^2}{2}$ and $g(r)=\frac{r^2}{2}$ inside  $r \le \epsilon$;
    \item $f'(r)\le 0$ and $g'(r) \ge 0$ for all $r$; and
    \item $g(r)=1$ and $f(r)=-r$ for $r \ge 2\epsilon$.
\end{itemize}
Moreover, using the symplectization coordinates near $\partial P$, for each fixed $\theta$,
$$ (r,\theta,p) \mapsto (e^{f(r)},p,\theta) \in (-\infty,0) \times \partial P \times S^1 \hookrightarrow (P \setminus \partial P) \times S^1$$
gives the trivial fibration near the binding induced by the open book.
\item On the mapping cylinder: The contact form on the mapping cylinder $((P \setminus \partial P) \times [0,2\pi])/\sim$ given by the complement of the binding is induced by the contact form 
\[
    \eta + dG+g\,d\theta,
\]
where $G,g\colon (P \setminus \partial P) \times [0,2\pi]\to \R$ are smooth functions, $G$ is compactly supported, and $g$ is positive.
\end{itemize}
\end{definition}

 \begin{lemma}\label{lemma_contact_form}
     Any one-form satisfying the properties in the above definition is automatically a contact form. Moreover, the Reeb vector field $\mathcal{R}_\alpha$ is transverse to the pages away from the binding while it is tangent to the binding, and moreover given by a positive multiple of 
     $${\mathcal{R}}_{{\alpha}_{\eta}}-\frac{e^{f(r)}f'(r)}{g'} \partial_\theta$$
     in the neighborhood $D^2_{3\epsilon} \times \partial P$ of the binding.
 \end{lemma}
 \begin{proof}
 Since the pull-back of the contact form to each page is a Liouville form, contact condition is easy to verify. The computation of the Reeb vector field can be verified.
 \end{proof}

\begin{example}\label{ex:trivial_ob}
Let $(P,\eta)$ be an exact symplectic manifold with boundary, where
$\eta$ is a Liouville form and the associated Liouville vector field
points outward along $\partial P$. Consider the identity monodromy
\(
    \phi=\operatorname{id}_P.
\)
The corresponding open book
\(
    (Y,\xi)=\operatorname{OB}(P,\eta,\operatorname{id}_P)
\)
can be constructed in the following way. Since the monodromy is the identity, the mapping torus is
\[
    M_{\operatorname{id}}
    =
    \frac{P\times[0,2\pi]}{(x,2\pi)\sim(x,0)}
    \cong P\times S^1.
\]
Thus, away from the binding, the open book is given by
\(
    (P\setminus \partial P) \times S^1.
\)
The boundary of the mapping torus is $\partial P\times S^1$, and
the closed manifold $Y$ is obtained by gluing in
$\partial P\times D^2$:
\[
    Y
    =
    (P\times S^1)
    \cup_{\partial P\times S^1}
    (\partial P\times D^2),
\]
where the pages are given by $P_\theta=P\times\{\theta\}$, while the binding is
\(
    B=\partial P\times\{0\}.
\)

Near the binding, the contact form can be chosen in the standard
form as in the definition
\[
    e^{f(r)}\alpha_\eta+g(r)\,d\theta,
\]
where $\alpha_\eta=\eta|_{\partial P}$ and $r$ is the radial coordinate
in the $D^2$-factor. 

Away from the binding, we can choose the contact form to be given by simply $\eta+d\theta,$ i.e. $G=0$. In other words, each page can be identified with $(P,\eta)$ for a fixed choice of Liouville form.

If we want to change the Liouville form on the pages to $\eta+dG$ for some compactly supported $G \colon P \to \R$, then it suffices to deform each page above by shifting it in the $\theta$-direction using the function $G$.
\end{example}

\begin{example}
As a concrete example, take
\(
    P=D^{2n}
\)
with its standard Liouville form
\(
    \eta
    =
    \frac{1}{2}
    \sum_{j=1}^n
    (x_j\,dy_j-y_j\,dx_j).
\)
Then
\[
    \operatorname{OB}(D^{2n},\eta,\operatorname{id})
    \cong S^{2n+1},
\]
and the resulting contact structure is the standard contact structure
on $S^{2n+1}$. Hence, the standard contact sphere provides the
simplest example of a contact manifold supported by a trivial open
book.
\end{example}
There are different possible choices and conventions. We will be using the following construction.

The contact manifold $\mathrm{OB}({P},\eta,\phi)$ is obtained in the following manner. The page $({P},\eta)$ has a collar on which there is a Liouville form preserving identification with the half symplectization
$$ ((-\infty,0] \times \partial{{P}},e^t\alpha_\eta),$$
where $\alpha_\eta=\eta|_{T\partial{{P}}}$ is the induced contact form on the boundary. We use
$$ {P}_{<a} \subset {P}$$
to denote the open sub-domain ${P} \setminus ([a,0] \times \partial{{P}})$.
\begin{itemize}
\item \emph{The neighborhood of the binding} is the contact manifold
$$(\partial {P}\times D_{3\epsilon}^2,e^{f(r)}\alpha_\eta+g(r)d\theta)$$
where $r$ and $\theta$ are polar coordinates on $D_{3\epsilon}^2$. Here the smooth non-negative functions $f$ and $g$ satisfy the following properties:
\begin{itemize}
    \item $f(r)=-\frac{r^2}{2}$ and $g(r)=\frac{r^2}{2}$ inside  $r \le \epsilon$;
    \item $f'(r)\le 0$ and $g'(r) \ge 0$ for all $r$; and
    \item $g(r)=1$ and $f(r)=-r$ for $r \ge 2\epsilon$.
\end{itemize}
\item In the case of non-trivial monodromy, \emph{the union of the pages} is obtained by forming the mapping cylinder of the monodromy and interpolating between pages
$$({P}_{<-2\epsilon} \times S^{1}_\theta,\eta+d\theta).$$ 
\item The contact manifold $\mathrm{OB}({P},\eta,\phi)$ defined by the abstract open book is finally obtained  by gluing the two contact manifolds above via the identification of the neighborhood of the binding
$$(y,(r,\theta))\in \partial {P}\times D_{3\epsilon}^2$$
for $r \ge 2\epsilon$ and
$$(f(r),y,\theta) \in [-3\epsilon,-2\epsilon) \times \partial{{P}}\times S^1_\theta \hookrightarrow ({P}_{<-2\epsilon}\setminus {P}_{<-3\epsilon})\times S^1_\theta.$$
\end{itemize}

We show that there exist methods to deform the pages of an open book decomposition, producing a new open book decomposition with modified pages and monodromy, while preserving the same contact manifold. Recall that by Gray's stability theorem, two isotopic contact structures are contactomorphic, \cite{geiges}.

\begin{lemma}\label{lemma_def_ob}
Let $(Y,\xi=\ker \alpha)$ be a contact manifold supported by an open book $\mathrm{OB}({P},\eta,\phi)$, and $H \colon P \setminus \partial P \to \R$ is an arbitrary compactly supported function. Then we can isotope the contact structure and provide a new contact form, while fixing a neighborhood of the binding, to get a contact structure $(Y,\xi'=\ker \alpha')$ supported by the same open book, where the mapping cylinder can be identified with 
$$ \big((P \setminus \partial P) \times [\theta_0,\theta_0+2\pi]/\sim, \eta+dH+dG+gd\theta\big),$$
where $G, g\colon (P \setminus \partial P) \times [\theta_0,\theta_0+2\pi]\to \R $  and $G$ vanishes  and $g\geq0$ equals to 1 except possibly in an arbitrarily small neighborhood of $(P \setminus \partial P) \times \{\theta_0\}$. In other words, most of the pages are identified with $(P,\eta+dH)$. 
\end{lemma}

\begin{proof}
    Consider $1$-parameter family of one-forms
$$\eta+\rho s\,dH+dG+gd\theta$$
    where $\rho$ is a bump function and $s\in [0,1]$. By Lemma \ref{lemma_contact_form}, this is a $1$-parameter family of contact forms. In addition, open book remains compatible with all the contact forms in the $1$-parameter family.
\end{proof}
Having established our conventions for open books, we now proceed to the notion of a \emph{symplectic Lefschetz fibration}. For the suspension construction, symplectic Lefschetz fibration plays a crucial role.
       
       Let $P$ be a compact $2n$-dimensional manifold with corners equipped with an exact symplectic form $\omega=d\eta$ such that the both faces of the  boundary $\partial {P}=\partial_v {P}\cup \partial_h {P}$ are convex.
A \textit{compatible Lefschetz fibration} on a compact $2n$-dimensional manifold with corners $P$ is a holomorphic map $ \pi: {P}^{2n}\rightarrow D^2$ for some choice of compatible complex structure on ${P}$, which satisfies the following conditions:

\begin{enumerate}
\item The map $\pi$ has finitely many non-degenerate distinct critical values $s_1,...,s_k \in int(D^2)$, and there exists an unique critical point $r_j \in  \pi^{-1}(s_j)$ for each $j=1,...,k$,
and the compatible complex structure on ${P}$ is integrable near the critical points,

\item Near each critical point $r_j$ and the corresponding critical value $s_j$, there are local complex coordinate charts matching with the orientations of ${P}$ and  $D^2$ such that the map $\pi$ locally has the form $\pi(z_1,...,z_n)=z_1^2+...+z_n^2$,

\item The restriction of $\pi$ to ${P}\backslash \pi^{-1}(\{s_1,...,s_k\})$ is a locally trivial fibration  over $D^2\setminus \{s_1,...,s_k\}$ whose fibers are $2n-2$-dimensional exact symplectic manifolds with convex boundary for the restricted forms,

\item The restriction $\pi|_{\partial_v {P}}: \partial_v {P}\rightarrow \partial D^2$ to the vertical boundary  is a surjective smooth fiber bundle. Moreover, there is a neighborhood of the horizontal boundary $\partial_h {P}=\bigcup\limits_{z\in D^2} \partial(\pi^{-1}(z))$ such that the restriction of $\pi$ to this neighborhood is a product fibration $D^2\times \mathcal{N}(\partial F)$ where $F$ is a regular fiber of $\pi$ and $ \mathcal{N}(\partial F)$ is a neighborhood of $\partial F$. The restricted Liouville form on this product decomposes as a sum of forms from the two factors, ensuring compatibility with the fiberwise structure. It follows that, $\pi|_{\partial_h {P}}: \partial_h {P}\rightarrow D^2$ is a surjective fiber bundles.
\end{enumerate}

The corners of ${P}$ can be rounded to get a Liouville domain whose completion has an exact symplectic Lefschetz fibration which is obtained by extending the compatible Lefschetz fibration.

The open book decomposition on the boundary of a Liouville manifold with corners induced by a compatible Lefschetz fibration on the manifold becomes apparent in the following observation.

  It is a well known fact that  a compatible Lefschetz fibration on a Liouville domain with corners, as defined above, induces an open book decomposition on its boundary. The map $$\pi|_{\partial_v P} : \partial_v P \to \partial D^2,$$
appearing in condition (4) of the definition of a compatible Lefschetz fibration, is obtained by restricting $\pi$ after choosing an orientation-preserving identification of $\partial D^2$ with $S^1$. Since there are no critical values on $\partial D^2$ and since the fibration is trivial near the horizontal boundary, the vertical boundary $\partial_v P$—that is, the union of regular fibers over $\partial D^2$—is a smooth fiber bundle over $S^1$, which provides the pages of the induced open book decomposition.

Similarly, the condition on the horizontal boundary $\partial_h P$ describes a neighborhood of the binding. The binding of the induced open book can be identified with the boundary of a central regular fiber, which is common to all pages.

\begin{example}\label{ex:induced_ob}
    Consider the trivial Lefschetz fibration $$
\operatorname{proj}_\C \colon \hat{P} \times \mathbb{C} \to \mathbb{C},
$$
on the Liouville manifold $\bigl(\hat{P} \times \mathbb{C},\, \lambda\bigr).$ Since the Liouville vector field associated to $\lambda$ is complete, the ideal boundary $\partial_\infty(\hat{P} \times \mathbb{C})$ carries a natural contact structure given by the restriction of $\lambda$. Choosing an exhausting function $g \colon \hat{P} \to \mathbb{R}_{\ge 0}$, the boundary at infinity may be identified with the level set
$$
\partial_\infty(\hat{P} \times \mathbb{C}) \simeq \left\{ (x,z) \in \hat{P} \times \mathbb{C} \mid g(x) + \tfrac{|z|^2}{2} = R \right\}
$$
for some fixed $R \gg 0$.

The Lefschetz fibration $\operatorname{proj}_\C$ induces a compatible open book decomposition on $\partial_\infty(\hat{P} \times \mathbb{C})$ as follows. The binding is given by
$$
B := \partial_\infty(\hat{P} \times \mathbb{C}) \cap \{z = 0\} \cong \partial_\infty \hat{P},
$$
and the map
$$
\frac{z}{|z|} \colon \partial_\infty(\hat{P} \times \mathbb{C}) \setminus B \to S^1
$$
defines the open book. The pages are the closures of the fibers
$$
{\operatorname{proj}_\C}^{-1}(R_{>0} \cdot e^{i\theta}) \cap \partial_\infty(\hat{P} \times \mathbb{C}),
$$
which are naturally identified with $P$. The monodromy of this open book is trivial and each page can be identified with $(P,{\eta})$ for the pull-back of the induced contact form.
\end{example}

\section{Pre-Lagrangians and pre-Lagrangian fillings}
To study fillability and classifications of Legendrians we will need to develop the notion of pre-Lagrangians and pre-Lagrangian fillings in contact manifolds. The definition of pre-Lagrangians itself is standard \cite{eliashberg1995lagrangian}, while the notion of pre-Lagrangian filling is new.
\begin{definition}
A closed $n+1$-dimensional submanifold $\mathcal{L} \subset (Y,\xi)$ of a contact manifold is \emph{pre-Lagrangian} if it admits a Lagrangian lift to the symplectization $(\R \times Y,d(e^t\alpha))$ or, equivalently, if there is a smooth function $G \colon \mathcal{L} \to \R$ for which the one-form $e^G(\alpha|_{T\mathcal{L}})$ is closed.
\end{definition}
Recall that a pre-Lagrangian $\mathcal{L}$ carries a hyper-plane distribution $\ker (\alpha|_{T\mathcal{L}}) \subset T\mathcal{L}$ that defines a non-singular foliation, so called \emph{characteristic foliation}. When considering pre-Lagrangians with boundary, in the context of Lagrangian fillability of Legendrians, it is convenient to relax the condition of being pre-Lagrangian along the boundary and also allow singularities there. The following definition of a type of pre-Lagrangians is inspired by the work of Courte--Ekholm in \cite{CE}.
\begin{definition}
An $n+1$-dimensional submanifold $\mathcal{L} \subset (Y^{2n+1},\xi)$ with Legendrian boundary $\Lambda$ of a $2n+1$-dimensional contact manifold is an \emph{exact pre-Lagrangian filling} of $\Lambda$ if there exists a smooth function
    $$G \colon \mathcal{L} \setminus  \partial\mathcal{L} \to \R$$ 
     which satisfies the property that

    \begin{itemize}
        \item $e^{G}(\alpha|_{T\mathcal{L}})$ is exact on $\mathcal{L} \setminus  \partial\mathcal{L}$.
        \item For some smooth identification
        $$(-\epsilon,0]_s \times \partial \mathcal{L}\hookrightarrow \mathcal{L}$$
        of a collar neighborhood of the boundary of $\mathcal{L}$, we can write $G=\tilde{G}+H$ for some smooth function $H\colon \mathcal{L}\to \R$, where the function $\tilde{G}=G-H$ only depends on the collar coordinate $s$, and satisfies $$\lim_{s \to 0} \tilde{G}(s)=\lim_{s \to 0} \tilde{G}'(s)=+\infty.$$
    \end{itemize} 
\end{definition}
It is important to note that the exact pre-Lagrangian filling does not necessarily satisfy the pre-Lagrangian condition along the boundary. In fact, as we will see, the pre-Lagrangian fillings that we are interested here satisfy the properties of Definition \ref{def:halfbook} below, which implies that their characteristic foliations typically have singularities in the boundary. This means that these submanifolds only are pre-Lagrangian in the standard sense in their interiors. The definition is motivated by the following result, by which an exact pre-Lagrangian filling can be deformed into an exact Lagrangian filling of its Legendrian boundary.

\begin{proposition}
\label{prop:lagfilling}
    An exact pre-Lagrangian filling $\mathcal{L} \subset (Y,\xi)$ of the Legendrian $\Lambda$ admits a smooth lift to the half symplectization $((-\infty,0] \times Y,d(e^t\alpha))$ with boundary in $\{0\} \times Y$ that can be perturbed to an embedded exact Lagrangian inside $(-\infty,0] \times Y$ with boundary contained in $\{0\} \times Y$, which is cylindrical near the boundary.
\end{proposition}
\begin{proof}
For any $A \gg 0$ we obtain a lift $\{t=G\}$ to the symplectization of the subdomain
$$\mathcal{L}_{\le A} = G^{-1}(-\infty,A] \subset \mathcal{L}$$ 
which is an exact Lagrangian with boundary contained in $\{A\} \times Y$, and normal transverse to the latter contact type hypersurface. For any $A  \gg 0$ the boundary of $\mathcal{L}_{\le A}$ is not necessarily Legendrian. However, it becomes arbitrarily $C^\infty$-close to the Legendrian $\Lambda$ as $A \to +\infty$. In addition, because of the assumption on the derivative of $G'$ it follows that $$\mathcal{L}_{\le A} \: \cap \: [A-1,A] \times Y$$
becomes arbitrarily $C^\infty$-close to the trivial Lagrangian cylinder
$$ [A-1,A] \times \Lambda \subset \R \times Y$$
over $\Lambda$.

Using Weinstein's Lagrangian standard neighborhood theorem we write $\mathcal{L}_{\le A}$ as the section of a $C^\infty$-small exact one-form in the cotangent bundle
$$D^*_\epsilon ([A-1,A] \times \Lambda) \hookrightarrow [A-1,A] \times Y.$$
Since the one-form is exact with an arbitrarily $C^\infty$-small primitive, say $\mathcal{L}_{\le A}$ is the graph of $dh$, we can use a bump function $\rho$ to cut off the primitive, deforming $\mathcal{L}_{\le A}$ to coincide with the graph of $d(\rho h)$ which is cylindrical Legendrian near the boundary if $\rho$ vanishes there.
\end{proof}

The following type of exact pre-Lagrangian filling will be of particular importance.
\begin{definition}
\label{def:halfbook}
An $n+1$-dimensional submanifold $\mathcal{L} \subset (Y^{2n+1},\xi)$ of a $2n+1$-dimensional contact manifold with non-empty Legendrian boundary has a characteristic foliation $\ker (\alpha|_{T\mathcal{L}})$ of \emph{half open book-type} if the following holds:
\begin{itemize}
\item the characteristic foliation is tangent to the boundary (possibly with singularities), while the foliation is non-singular in the interior.
\item doubling $\mathcal{L}$ along its boundary produces a closed manifold with an open book decomposition, where the foliation is identified with the singular foliation produced by the pages, and $\mathcal{L}$ is given by half of the open book bounded by a union of two pages.
\end{itemize}
In particular, there is a diffeomorphism
$$\mathcal{L} \setminus \partial \mathcal{L} \cong \mathcal{P} \times (0,\pi)$$
of the interior, where $\mathcal{P}$ is the interior of the pages of the open book, and where each $\mathcal{P} \times \{s\}$ is a Legendrian leaf of the foliation.
\end{definition}

\begin{proposition}\label{prop: pre-lag}
    If $\mathcal{L}$ has characteristic foliation of half open book-type, then $\mathcal{L}$ is an exact pre-Lagrangian filling.
\end{proposition}
\begin{proof}
By assumption there exits one-form that defines the singular foliation by pages of the open book. Consider the one-form $\beta$ that is given by  $d\theta$  away from the neighborhood of the binding  and by $g(r)d\theta$ near the binding, where $g(r)=1$ for $r \ge 2\epsilon$, and $g(r)=r^2/2$ for $r \le \epsilon$, and otherwise non-vanishing. By assumption, the one-form defining the characteristic foliation $\alpha|_{T\mathcal{L}}$ satisfies $\beta=e^H\alpha|_{T\mathcal{L}}$ for some smooth function $H \colon \mathcal{L} \to \R$. Note that $e^{-\log g(r)} \beta=d\theta$ is a closed one-form in the interior of the pages, but that the form is non-smooth precisely along the binding.

We may assume that $\partial\mathcal{L}$ is given as the union of the two pages $\{\theta=0,\pi\}$ including the binding.

To construct the sought function $G \colon \mathcal{L} \setminus \partial \mathcal{L} \to \R$ we start by considering $\tilde{G}=-\log g(r)+\varphi(\theta)$ where $\varphi \colon [0,\pi] \to \R_{\ge 0}$ satisfies the property that $\varphi(\theta)$ tends to $+\infty$ as $\theta \to 0,\pi$, while $\varphi'(\theta)$ tends to $-\infty$ as $\theta \to 0$ and to $+\infty$ as $\theta \to \pi.$ We may finally take $G=\tilde{G}+H$ as our sought function, which thus has the property that $e^G\alpha|_{T\mathcal{L}}=d\theta$ away from the binding.
\end{proof}

We will show that it is easy to construct a pre-Lagrangian whose characteristic foliation is of half open book-type, by starting with an Legendrian embedding of a single page. Since the page is a Legendrian with non-empty boundary, and since such Legendrians satisfy an $h$-principle \cite{CieliebakEliashbergMishachev}, there are plenty of such embeddings to start from, and they are determined by their classical invariants. First we recall the definition of an Legendrian ambient surgery from Dimitroglou Rizell's work \cite{Dimitroglou:Ambient}, which deforms a Legendrian $\Lambda \subset (Y,\xi)$ by performing an ambient $k$-surgery using an embedded isotropic surgery $k+1$-disk $D^{k+1} \subset (Y,\xi)$ that intersects $\Lambda$ only along its boundary. In addition there is a standard Lagrangian $k+1$-handle cobordism from the original Legendrian $\Lambda$ to the Legendrian $\Lambda_D$ resulting from the Legendrian ambient surgery.   

It is not difficult to construct an exact Lagrangian filling in $J^1(\overline{P} \times \R)$ given by a front consisting of the union over two sheets $\pm j^1F$ joining together at a cusp edge which is smoothly diffeomorphic to $\overline{P}$. Such a filling can be seen to be decomposable. Below we show how to construct exact pre-Lagrangian open books whose Legendrian boundary admits fillings of this type.
\begin{theorem}\label{thm:existence} 
    Any Legendrian embedding $\overline{\mathcal{P}}^n \hookrightarrow (Y^{2n+1},\xi)$ of a smooth manifold with boundary can be extended to the page of an exact pre-Lagrangian filling of half open book-type inside a standard jet-neighborhood $J^1\overline{\mathcal{P}}$ of $\overline{\mathcal{P}} \subset Y$.

    Furthermore, this Legendrian admits a regular exact Lagrangian filling inside $\R \times J^1\overline{P} \cong T^*(\overline{P} \times \R)$, moreover, is flexible when $n >1$. This filling is in fact decomposable and consists of standard Lagrangian handles of index $<n$. 
\end{theorem}

\begin{proof}
We start with the construction of the boundary $\Lambda \cong \overline{\mathcal{P}} \sqcup_{\partial \overline{\mathcal{P}}} \overline{\mathcal{P}}$ of the pre-Lagrangian, by declaring that its image under the front projection $J^1\overline{\mathcal{P}} \to \overline{\mathcal{P}} \times \R_z$ is given by a union of two sheets $j^1(\pm F)$ that are joined along a cusp edge that is equal to the boundary of $\overline{\mathcal{P}}$ at $\{z=0\}$, and whose front has no other singularities. In particular, the function $F \colon \overline{\mathcal{P}} \to \R_{\ge 0}$ vanishes precisely along the boundary of $\overline{\mathcal{P}}$ where it behaves like $x^{3/2}$ for a coordinate $x$ defined on the collar of $\overline{\mathcal{P}}$ and is positive in the interior $\overline{\mathcal{P}} \setminus \partial \overline{\mathcal{P}}$.

The pre-Lagrangian with characteristic foliation of half open book-type can be constructed by starting with the one-parameter family $s\Lambda$ of smooth Legendrians for $s \in (0,1]$ that is equal to $\Lambda$ for $s=1$ and which degenerates onto $\overline{\mathcal{P}}$ for $s=0$. Note that, even though this union is a smooth $n+1$-dimensional manifold with boundary $\Lambda$, whose characteristic leaves corresponding to the Legendrian sheets, but that the characteristic foliation is not of half open book-type. The reason is that the family of sheets does not have the correct behavior near the cusp edge where they meet. To amend this, we deform the Legendrian leaves in an arbitrarily small neighborhood near the cusp edge while fixing the Legendrians $j^1(\pm s F)$. This can be done by considering the following model.

First, consider the foliation of a subset of $\R^2$ given by the graphs $\frac{d}{dx}\pm s\cdot x^{3/2}$, $s \in [0,1]$, $x \ge 0$. Deform the leaves in the interior of the domain foliated by the graphs so that the singular foliation becomes diffeomorphic to a family of linear rays through the origin that foliates the right half-plane. The two different singular foliations on the right half-plane are shown in Figure \ref{fig:def_sF}. This can be done in an arbitrarily small neighborhood of the singularity, i.e.~the origin. Then integrate the deformed foliation to a new family of functions, and use the one-jets of these deformed functions to build the sought exact pre-Lagrangian filling $\mathcal{L}$ of half open book-type. Note that the binding of this half open book is given by the cusp edge $\partial \overline{\mathcal{P}}$, while the closure of the pages are given by the Legendrian sections $j^1(\pm F_s) \cong \overline{\mathcal{P}}$ where $F_s=s\cdot F$ away from some arbitrarily small neighborhood of the binding.

\begin{figure}[ht]
    \centering
    \scalebox{.7}{
    \begin{minipage}[c]{0.4\textwidth}
\centering
\begin{tikzpicture}
  \draw[black] (0, -3) -- (0, 3);
  \draw[black] (0, 0) -- (4, 0);
  \draw[black]   plot[smooth, domain=-sqrt(4):sqrt(4)] ( \x * \x, \x);
  \draw[black]   plot[smooth, domain=-sqrt(8):sqrt(8)] (0.5 * \x * \x, \x);
  \draw[black]   plot[smooth, domain=-sqrt(2):sqrt(2)] (2 * \x * \x, \x);
  \draw[black]   plot[smooth, domain=-sqrt(1):sqrt(1)] (4 * \x * \x, \x);
  \end{tikzpicture}
\end{minipage}
}
\hspace{3\baselineskip}
\scalebox{.7}{
\begin{minipage}[c]{0.4\textwidth}
\centering
\begin{tikzpicture}
   \draw[black] (0, -3) -- (0, 3);
  \draw[black] (0, 0) -- (4, 0);
 \draw[black, domain=0:3] plot (\x, {0.9*\x});
 \draw[black, domain=0:4] plot (\x, {0.4*\x});
 \draw[black, domain=0:2] plot (\x, {1.5*\x});

  \draw[black, domain=3:0] plot (\x, {- 0.9*\x});
  \draw[black, domain=4:0] plot (\x, {-0.4*\x});
   \draw[black, domain=2:0] plot (\x, {-1.5* \x});
\end{tikzpicture}
\end{minipage}
}
\caption{Two non-diffeomorphic singular foliations. On the left, all leaves are tangent at the origin while, on the right, the singular foliation is of half open book-type. Note that they can be made diffeomorphic after an arbitrarily small deformation near the singularity.}
\label{fig:def_sF}
\end{figure}
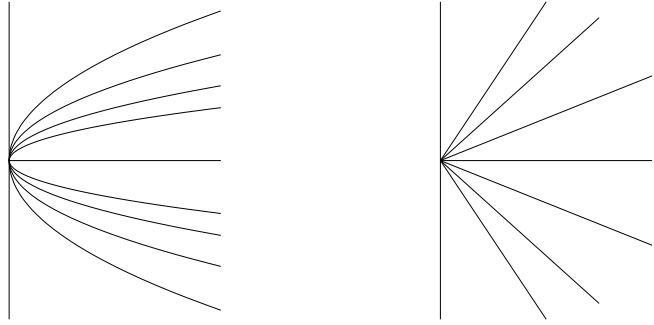

We then show that the Legendrian $\Lambda$ can be obtained by a sequence of ambient Legendrian $k$-surgeries for $k<n-1$. Start by choosing a Morse function $f\colon \overline{\mathcal{P}} \to \R_{\ge 0}$ whose positive gradient is inwards pointing along the boundary, and whose critical points have Morse index at least $1$. We may assume that $F$ and $f$ coincide away from some arbitrarily small neighborhood of the boundary, and that their critical points are the same inside the interior of the page $ \overline{\mathcal{P}} \setminus \partial \overline{\mathcal{P}}$.

We start by placing small standard fillable Legendrian spheres around the maxima of $f$. One can then perform Legendrian ambient $k$-surgeries as defined in \cite{Dimitroglou:Ambient} on this Legendrian, by using the isotropic surgery disks contained inside the zero section $\overline{\mathcal{P}}=j^10 \subset J^1\overline{\mathcal{P}}$ that are given by the stable manifolds of the positive gradient flow of $f$. More precisely, we inductively perform ambient surgeries on the disks with boundary on the Legendrian produced in the previous stage of the process, and then continue. Note that the ambient surgery disks all intersect the cusp edge of the Legendrian produced  in a fattened sphere. Since the Legendrian resulting from the ambient surgery construction consists of union of two 1-jets glued along cusp edge, one can readily check that the obtained Legendrian is Legendrian isotopic to the original $\Lambda$, which also consists of the union of two sheets. See Figure \ref{fig:ambient_surgery} for the surgery in the case of a Legendrian surface.

The Lagrangian standard handle-attachments defined by the surgery build a decomposable exact Lagrangian filling of $\Lambda$ which is diffeomorphic to the exact pre-Lagrangian filling $\mathcal{L}$.
\end{proof}

\begin{figure}[ht]
    \centering
    \begin{minipage}[c]{0.4\textwidth}
\centering
\begin{tikzpicture}[scale=1.2]

\def\basecurve(#1,#2){
    plot[domain=-1:1, samples=200, smooth]
    (\x, {#1*(1-\x*\x)*sqrt(1-\x*\x)})
}

\foreach \y in {1,0.5} {
    \draw[line width=0.5pt]
       \basecurve(\y,0);
    \draw[line width=0.5pt]
        plot[domain=-1:1, samples=200, smooth]
        (\x, {-\y*(1-\x*\x)*sqrt(1-\x*\x)});
}
\draw[red, dashed, thick] (0,-0.5) arc (-90:90:3 and 0.5);
\end{tikzpicture}
    \label{fig:leg_sphere}
\end{minipage}
\hspace{0.5\baselineskip}
     \begin{minipage}[c]{0.5\textwidth}
\centering
    \begin{tikzpicture}[scale=1]

\def\basecurve(#1,#2){
    plot[domain=-1:1, samples=200, smooth]
    (\x, {#1*(1-\x*\x)*sqrt(1-\x*\x)})
}

\foreach \y in {1,0.5} {
    \draw[line width=0.5pt]
       \basecurve(\y,0);
    \draw[line width=0.5pt]
        plot[domain=-1:1, samples=200, smooth]
        (\x, {-\y*(1-\x*\x)*sqrt(1-\x*\x)});
}
\begin{scope}[xshift=3cm]
\foreach \y in {1,0.5} {
    \draw[line width=0.5pt]
       \basecurve(\y,0);
    \draw[line width=0.5pt]
        plot[domain=-1:1, samples=200, smooth]
        (\x, {-\y*(1-\x*\x)*sqrt(1-\x*\x)});
}
\end{scope}
\draw (1.5,0) ellipse (0.5 and 0.4);
\draw (1.5,0) ellipse (2.5 and 1.5);
\draw[red, dashed, thick] (0,0.5) arc (90:-90:3.5 and 0.5);
\end{tikzpicture}
    \end{minipage}
    \caption{Left: The standard Legendrian surgery sphere with red surgery arc. Right: The resulting manifold after  Legendrian ambient k-surgery.}
    \label{fig:ambient_surgery}
\end{figure}
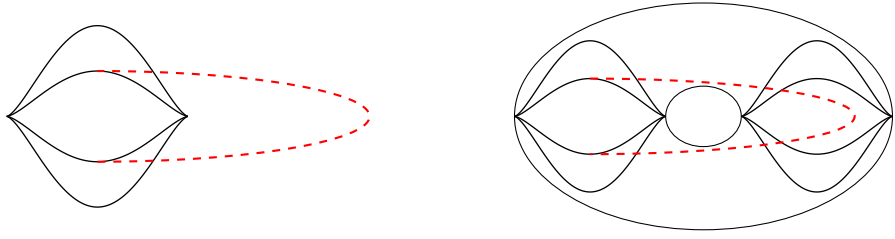

\begin{remark}
    We believe that it is possible to show that the filling constructed in Theorem \ref{thm:existence} is Hamiltonian isotopic to the one produced by Proposition \ref{prop:lagfilling}. Since we do not need this fact, we will postpone the proof to later work.
\end{remark}

\begin{figure}[ht]
\centering
\begin{minipage}[c]{0.25\textwidth}
\centering
\rule{2cm}{0.6pt}
\end{minipage}
\hspace{0.5cm}
\begin{minipage}[c]{0.5\textwidth}
\centering
\begin{tikzpicture}[scale=1.2]

\def\basecurve(#1){
    plot[domain=-1:1, samples=200, smooth]
    (\x, {#1*(1-\x*\x)*sqrt(1-\x*\x)})
}

\foreach \y in {1,0.70,0.40,0.10} {
    \draw[line width=0.5pt] \basecurve(\y);
    \draw[line width=0.5pt]
        plot[domain=-1:1, samples=200, smooth]
        (\x, {-\y*(1-\x*\x)*sqrt(1-\x*\x)});
}

\end{tikzpicture}
\end{minipage}
\label{fig:leg_foliation}
\vspace{2\baselineskip}
\centering
\begin{minipage}[c]{0.25\textwidth}
\centering
\begin{tikzpicture}
  \fill[gray!40, opacity = 0.3] (0,0) ellipse (1 and 0.6);
  \draw (0,0) ellipse (1 and 0.6);
\end{tikzpicture}
\end{minipage}
\hspace{0.3cm}
\begin{minipage}[c]{0.5\textwidth}
\centering
\begin{tikzpicture}[scale=1.2]

\def\basecurve(#1,#2){
    plot[domain=-1:1, samples=200, smooth]
    (\x, {#1*(1-\x*\x)*sqrt(1-\x*\x)})
}

\foreach \y in {1,0.5,0.25,0} {
    \draw[line width=0.5pt]
       \basecurve(\y,0);
    \draw[line width=0.5pt]
        plot[domain=-1:1, samples=200, smooth]
        (\x, {-\y*(1-\x*\x)*sqrt(1-\x*\x)});
}
\draw (0,0) ellipse (1 and 0.6);
\end{tikzpicture}
    \label{fig:leg_foliation_ufo}
\end{minipage}
   \centering
\begin{minipage}[c]{0.25\textwidth}
\centering
\vspace{2\baselineskip}
\begin{tikzpicture}
\fill[gray!40, opacity = 0.3, even odd rule] (1.5,0) ellipse (0.8 and 0.4) (1.5,0) ellipse (2 and 1);
\draw (1.5,0) ellipse (0.8 and 0.4);
\draw (1.5,0) ellipse (2 and 1);
\end{tikzpicture}
\end{minipage}
\hspace{0.3cm}
   \begin{minipage}[c]{0.5\textwidth}
\centering
\vspace{2\baselineskip}
    \begin{tikzpicture}[scale=1]

\def\basecurve(#1,#2){
    plot[domain=-1:1, samples=200, smooth]
    (\x, {#1*(1-\x*\x)*sqrt(1-\x*\x)})
}

\foreach \y in {1,0.5,0.10} {
    \draw[line width=0.5pt]
       \basecurve(\y,0);
    \draw[line width=0.5pt]
        plot[domain=-1:1, samples=200, smooth]
        (\x, {-\y*(1-\x*\x)*sqrt(1-\x*\x)});
}
\begin{scope}[xshift=3cm]
\foreach \y in {1,0.5,0.10} {
    \draw[line width=0.5pt]
       \basecurve(\y,0);
    \draw[line width=0.5pt]
        plot[domain=-1:1, samples=200, smooth]
        (\x, {-\y*(1-\x*\x)*sqrt(1-\x*\x)});
}
\end{scope}
\draw (1.5,0) ellipse (0.5 and 0.4);
\draw (1.5,0) ellipse (2.5 and 1.5);
\end{tikzpicture}
    \end{minipage}
    \caption{Left: Different choices of the Legendrian embedding $\overline{\mathcal{P}}$ in Theorem \ref{thm:existence}. Right: Characteristic foliation of half open book-type corresponding to the choices of $\overline{\mathcal{P}}$ on the left.}
    \label{fig:leg_foliation_torus}
\end{figure}
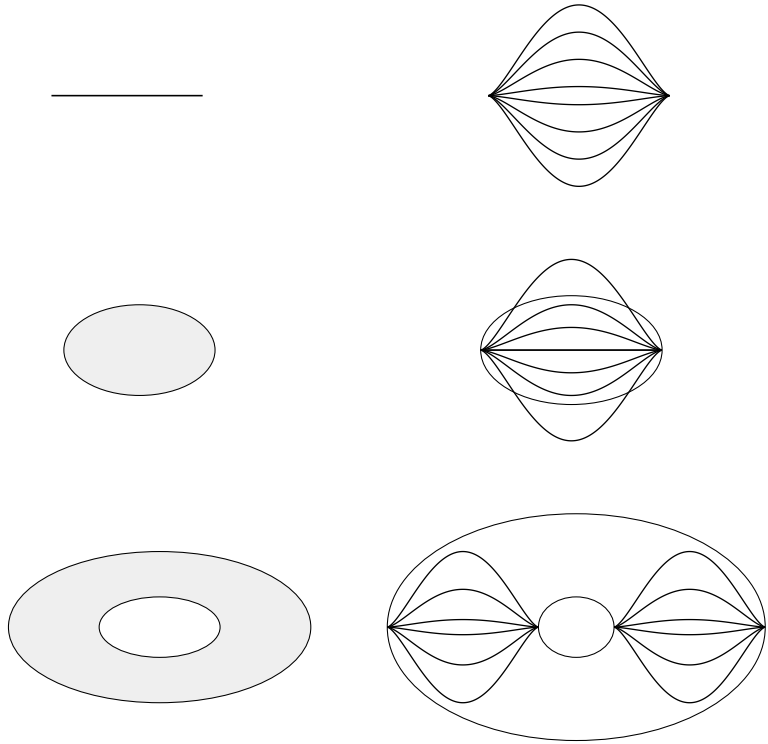

 Recall the definition of a flexible regular Lagrangian fillings in the sense of Eliashberg--Ganatra--Lazarev from \cite{EliashbergGanatra}. By Theorem 4.2 in the same paper, the symplectomorphism class of flexible Lagrangian fillings only depend on classical topological data. Here we establish a similar uniqueness result for the Legendrian which admit an exact pre-Lagrangian filling whose characteristic foliation is of half open book-type.

\begin{theorem}\label{theorem:flex}
    Let  $\Lambda$ and $\Lambda'$  be two Legendrians that both admit exact pre-Lagrangian fillings with characteristic foliations of half open book-type. If the pages of two half open book have the same classical invariants (homotopy class and rotation classes), then the two Legendrians are Legendrian isotopic by a contact isotopy.
    \end{theorem}
\begin{proof}
    The construction of Legendrian isotopy follows a similar strategy as the one in \cite{CE}.  The  crucial step is showing that the contact isotopy class of an exact pre-Lagrangian filling of half open book-type only depends on the Legendrian isotopy class of the Legendrian with boundary given by the  closure of a leaf $\overline{\mathcal{P}}$ of the characteristic foliation, i.e.~a page of the half open book. 
    
   When two pre-Lagrangian fillings of half open book-type  have the same classical invariants, then the $h$-principle for Legendrians with boundary implies that there is a contact isotopy that takes one page to the other, \cite{CieliebakEliashbergMishachev}. With a bit more work we can even ensure that two pre-Lagrangians coincide near one of the pages, say $\frac{\pi}{2}$. There is a family $\overline{\mathcal{P}}_{\pi-t}\cup \overline{\mathcal{P}}_{t}$, where $t\in [0,\frac{\pi}{2}+\epsilon]$, of singular Legendrians starting with two boundaries of two different open books  and both ending at the union $\overline{\mathcal{P}}_{\frac{\pi}{2}-\epsilon}\cup \overline{\mathcal{P}}_{\frac{\pi}{2}+\epsilon}$. Note that, for $t>0$ the corresponding Legendrian in the filling has a singularity near the binding since the two tangent planes make a sharp angle there. However, we can smoothen these two Legendrians near the binding to create the isotopy we want.
\end{proof}

\section{Legendrian double}\label{doubling process}
Here, we provide the doubling construction. Our construction starts from a contact manifold $(Y,\xi)$  with  compatible open book $({P},\eta,\phi)$ where the Liouville domain $({P},\eta)$ is the page of the open book and the map $\phi$ is the symplectomorphism of the page ${P}$ and also we have an exact Lagrangian filling $L$ of the Legendrian submanifold $\Lambda$ in the page ${P}$ of the open book, i.e. $\eta|_{(TL)}=dF$  for some real valued function $F$ of $L$. We assume that $L$ is cylindrical over a Legendrian $\Lambda$ near the boundary of $(P,\eta)$, which means that $F$ can be taken to vanish near the boundary. We obtain a hypersurface of $(Y,\xi)$ by forming the union of the pages ${P}_0$ and ${P}_\pi$ and  identifying them along their boundaries $\partial {P}_0=\partial {P}_\pi$, that is, along binding. This union is a hypersurface that we denote by ${P}^d$. The Legendrian double will be constructed inside this hypersurface below.
\begin{proposition}
    The hypersurface ${P}^d$ is convex hypersurface of $(Y,\xi)$ with dividing set $\partial {P}$.
\end{proposition}
\begin{proof}
 We  define a contact Hamiltonian function that is equal to a positive constant in a neighborhood of one page, $P_0$, and to the negative of the same constant in a neighborhood of the opposite page, $P_\pi$. In the binding region, the function interpolates smoothly between these values as a function of $x$. More precisely, in the neighborhood $(D^2_{3\epsilon}\times \partial P, e^{f(r)}\alpha_\eta+g(r)d\theta) $ we consider the autonomous contact Hamiltonian $H(x,y,p)=h(x)$ where 
\begin{itemize}
    \item $h'(x)\geq 0$,
    \item $h(-x)=-h(x)$,
    \item $h(x)=1$ for $x<\epsilon$,
    \item $h(x)=-1$ for $x>\epsilon$.
\end{itemize}

The contact vector field can be seen to be given by $$
H\mathcal{R}_\alpha+\left(-h'\frac{g}{e^f(g'+f'g)}\right)\left({\mathcal{R}}_{{\alpha}_{\eta}}-\frac{e^f}{g} \partial_\theta \right).$$

The associated contact Hamiltonian vector field is the unique contact vector field generated by this Hamiltonian, whose flow preserves the contact structure. In the interior of the pages, where $h'=0$, this vector field coincides with non-zero multiple of  the Reeb vector field. By Lemma \ref{lemma_contact_form}, this vector field transverse to the pages. In the binding region, it pushes the hypersurface $P^d$ in the $y$-direction.

It is immediate from the definition that the dividing set is precisely the binding $\partial P$.
\end{proof}

Our aim is to construct the double $L^d$ of the exact Lagrangian filling $ L$. For this purpose,  we form the union of the Lagrangians corresponding to  $ L$  in the two opposite  pages  ${P}_0$ and ${P}_\pi$ along their Legendrian boundary  $\Lambda$. Hence, this is a closed half-dimensional submanifold of the convex hypersurface $P^d$ that intersects the dividing set in a Legendrian. There are several ways to deform $L^d$ away from the dividing set to make it Legendrian everywhere. We will now proceed to one such deformation.

Note that, since $L\subset (P,\eta)$ is an exact Lagrangian filling, by definition it is cylindrical and thus strongly exact near the binding. The reason is that $L$ has a collar that is cylindrical over a Legendrian, then the double $L^d$ is smooth and Legendrian near the binding. Since  $\eta|_{(TL^d)}=dF$ for $F\colon L\to \R$ supported away from the boundary, the double $L^d$ created in this manner is exact Lagrangian in the union of two pages, but it is not necessarily Legendrian in the interior of the pages unless the Lagrangian filling $L\subset (P,\eta)$ is \emph{strongly exact}, i.e. unless $\eta$ pulled back to $L$ \underline{vanishes}.  In Lemma \ref{lemma_def_ob}, we fix this. After deforming the contact form in the interior of the pages as described in Lemma~\ref{lemma_def_ob} (this does not change the isotopy class of the contact structure), we may assume that the contact form pulls back to $\eta-dF$ on both pages $L_0\subset P_0$ and $L_\pi\subset P_\pi$.  Since \(L\) is strongly exact for the new Liouville form $\eta-dF$, the doubling construction now produces a Legendrian submanifold $L^d_{Leg}$ contained in two opposite pages of the open book decomposition.
\begin{definition}\label{definition:Leg_double}
    Assume that $(Y,\xi)$ is a contact manifold with a compatible open book with page $(P,\eta)$ that contains an exact Lagrangian filling $L$. The submanifold $L^d_{Leg}\subset P^d \subset (Y,\xi)$  which is diffeomorphic to $L \cup_{\partial L} L$ and contained in the union $P^d$ of the pages of the deformation of the open book described above is a Legendrian that is called the \emph{Legendrian double} of the exact Lagrangian filling $L\subset (P,\eta)$.
\end{definition}
\begin{remark}
Alternatively, when the open book is trivial as in Example \ref{ex:trivial_ob}, i.e. all pages have the same Liouville form $\eta$, then we can change the Liouville form by simply deforming the pages in the Reeb direction as explained in the same example.
\end{remark}

In the following we assume that  $L$ is strongly exact or, alternatively, that the Liouville form on the pages is deformed, as in Lemma \ref{lemma_def_ob} to make it Legendrian contained inside a different open book. Consider the following submanifold in the contact manifold $(Y,\xi)$ $$ \mathcal{L}=\bigcup_{\theta \in [0,\pi]} {L_\theta}$$
 where $L_\theta$ is an exact Lagrangian filling of the Legendrian submanifold $\Lambda$ contained in the corresponding page $P_\theta$.
\begin{lemma}\label{lemma:pre_lag}
   After suitable perturbation of $L$, the submanifold $\mathcal{L}$ admits a characteristic foliation of half open book-type. 
\end{lemma}

\begin{proof}
  Indeed, the submanifold $\mathcal{L}$ has a characteristic foliation given by $ker(\alpha|_{T\mathcal{L}})$. From this foliation, we obtain the following observations.
\begin{itemize}
    \item The characteristic foliation is tangent to the boundary $L_{Leg}^d$ of $\mathcal{L}$ since $ker(\alpha|_{T\mathcal{L}})=T\mathcal{L}$ on the boundary and non-singular in the interior.
    \item If $\mathcal{L}$ is doubled along its boundary, we obtain an open book  with pages $L$.
\end{itemize}

In particular, we have the diffeomorphism 
$$\mathcal{L} \setminus \partial \mathcal{L} \cong \mathring L \times (0,\pi)$$
of the interior and each $L\times \{\theta\}$ is a Legendrian leaf of the foliation.

The above observation shows that $\mathcal{L}$ has a characteristic foliation of half open book-type. Then, by Proposition \ref{prop: pre-lag}, it is an exact pre-Lagrangian filling of the Legendrian $L_{Leg}^d$. Moreover, $\mathcal{L}$ has a smooth lift to the half symplectization which can be perturbed to an embedded exact Lagrangian in $(-\infty,0]\times Y$ with cylindrical ends by Proposition \ref{prop:lagfilling}.
 \end{proof} 
    
We are now ready to prove Theorem \ref{theorem:flexible_filling}.

\begin{proof}[Proof of Theorem \ref{theorem:flexible_filling}]
Theorem \ref{thm:existence} establishes the desired result, as Lemma \ref{lemma:pre_lag} provides the exact pre-Lagrangian filling   $\mathcal{L}$ bounded by $L_{Leg}^d$.
\end{proof}

\begin{proof}[Proof of Theorem \ref{thm:flex_trivial_fib}]

We start by showing that the Legendrian ideal boundary at infinity of the complete Lagrangian filling $\hat{L} \times \R $ in $ \hat{P} \times \C$ is equal to the Legendrian double of $L \subset (P,d\eta)$. 
More precisely, since $\hat{L} \times \R $ is tangent to the Liouville flow outside of a compact subset, it has a well-defined ideal boundary at infinity $\partial_\infty(\hat{L} \times \R) \subset \partial_\infty (\hat{P}\times \C)$. If we identify $\partial_\infty (\hat{P}\times \C)$ with the contact-type hypersurface
$$\{g(p)+\tfrac{|z|^2}{2}=R\} \subset \hat{P}\times \C$$
described in Example \ref{ex:induced_ob}, we can identify this ideal boundary with the Legendrian submanifold given by the intersection
$$
\partial_\infty(\hat{L} \times \R) 
= \left\{ (x,z) \in \hat{L} \times \R \;\middle|\; g(x) + \tfrac{|z|^2}{2} = R \right\}
\subset \partial_\infty(\hat{P} \times \C)
$$
of $\hat{L} \times \R$ and the hypersurface. We claim that this Legendrian is equal to $L_{Leg}^d$. Indeed, the Legendrian is clearly contained inside two opposite pages of the compatible open book with trivial monodromy described in Example \ref{ex:induced_ob}.

Pictorially, this Legendrian arises from two types of ends. When $x$ approaches infinity in $\hat{L}$, the condition forces $z \to 0$, recovering a copy of $\partial_\infty (\hat{L})$ sitting in the binding. On the other hand, when $z \to \pm \infty$, the function $g(x)$ must remain bounded, so $x$ stays in a compact region of $L$, producing two ends corresponding to $z \to +\infty$ and $z \to -\infty$. Thus, $\partial_\infty(\hat{L} \times \mathbb{R})$ interpolates between $\partial_\infty (\hat{L})$ and two copies of $L$, and defines a Legendrian submanifold, which is, by definition, $L_{Leg}^d$.

We proceed to show that the filling is flexible. Start by considering a Weinstein neighborhood of $(\hat{L}\times\R)$ in $\hat{P}\times \C$ identified with
$$
T^*(\hat{L}\times\R)\cong T^*\hat{L}\times T^*\mathbb{R},
$$
under which the Lagrangian $\hat{L}\times \R$ corresponds to the zero section
$$
0_{\hat{L}}\times 0_{\R} \subset {{T^*}\hat{L}}\times T^*\R.
$$
The above symplectic manifold is actually a Weinstein sector and not a Weinstein domain, but we will be make the appropriate identifications with a Liouville domain in the following.

Note that $\hat{L}\times \R$ is a regular Lagrangian submanifold of $T^*(\hat{L}\times\R)$. Since the cotangent bundle of a connected manifold with boundary is subcritical, the complementary cobordism is trivial. Consequently, the regular Lagrangian $\hat{L}\times \R$ is flexible when considered inside the cotangent bundle $T^*(\hat{L}\times\R)$. 

Next we claim that this flexible Lagrangian filling is decomposable. Since its boundary is the Legendrian double, Theorem~\ref{thm:existence} implies that the boundary, i.e.~the Legendrian double, admits a regular flexible exact Lagrangian filling obtained by a sequence of standard Lagrangian handle attachments. By the $h$-principle \cite[Theorem 4.2]{EliashbergGanatra}, there is a symplectomorphism $$
\Phi\colon T^*(\hat{L}\times\mathbb{R})
\longrightarrow
T^*(\hat{L}\times\mathbb{R})$$
which identifies the latter filling with $\hat{L}\times \R$. 

The above symplectomorphism $\Phi$ can be assumed to be cylindrical outside of a compact subset, which means that it is of the form $(t,y) \mapsto (t+h(y),\phi(y))$ in the symplectization coordinates for a contactomorphism
$\phi \in \mathrm{Cont}(Y,\xi)$ satisfying $\phi^*\alpha=e^{-h}\alpha$.
Denoted by
\(\mathcal{L}_{\mathrm{dec}}\) the aforementioned decomposable Lagrangian filling of the Legendrian
\(\partial_\infty(\hat{L}\times\mathbb{R})\), which by construction is obtained using only
Lagrangian handles of index \(<n\). Since, by the construction in Theorem \ref{thm:existence}, the Lagrangian filling can be constructed in the symplectization of a page of the pre-Lagrangian of half open book-type, we can find a Hamiltonian isotopy that places this decomposable filling inside an arbitrarily small neighborhood of $\hat{L} \times (0,+\infty)$, i.e.~the symplectization of the standard contact jet-neighborhood of the half $L \subset L^d_{Leg}$ of the Legendrian double. We will still denote this filling by \(\mathcal{L}_{\mathrm{dec}}\).

For \(t\geq 0\), we let \(\mathcal{L}_{\mathrm{dec}}^t\) denote the translation of \(\mathcal{L}_{\mathrm{dec}}\) by $t$ in the symplectization coordinate of the cylindrical end. Then
$$
\mathcal{L}_t
:=
\Phi^{-1}\bigl(\mathcal{L}_{\mathrm{dec}}^t\bigr)
$$
defines an exact Lagrangian isotopy starting with
\(\mathcal{L}_0=\hat{L}\times\mathbb{R}\) and ending with the decomposable Lagrangian filling
\(\mathcal{L}_T=\Phi^{-1}(\mathcal{L}_{\mathrm{dec}}^T)\) whenever $t=T\gg0$ is sufficiently large (so that $\mathcal{L}_{\mathrm{dec}}^T$ is contained in the region where the symplectomorphism is cylindrical). Here we have used the fact that cylindrical symplectomorphisms preserve the property of being decomposable.

Note that this Lagrangian isotopy fixes the Lagrangian set-wise outside of a compact subset. Thus, a standard fact implies that there is a compactly supported Hamiltonian isotopy from $\hat{L} \times \R$ to the latter decomposable filling with only standard handles of index $<n.$ In addition, the decomposable filling $\mathcal{L}_T$ lives inside the symplectization of an arbitrarily small standard contact jet-neighborhood of the half $L \subset L^d_{Leg}$ of the Legendrian double. 

Since the the complementary cobordism $P \times \C \setminus T^*(\hat{L} \times \R)$ is trivial in some neighborhood of $\partial_\infty(\hat{L}\times \R)$ the cobordism $\mathcal{L}_T$ still lives inside the symplectization also after adding the complementary cobordism. In particular, it remains decomposable with only handles of index $<n$, and thus flexible regular, as sought. 
\end{proof}

\section{Pages in suspended open books from Legendrian doubles} \label{section_symp_ribbon}

The suspension construction is a way to modify the topology of a page of an open book decomposition to yield a new open book of the same contact manifold. Finding an explicit description of the page of the new open book can be conveniently done using the Legendrian double construction, as we show in Theorem \ref{prop_stab_ob} below. To that end, we have to start by giving an explicit description of symplectic ribbons of Legendrian doubles.

A symplectic ribbon of a Legendrian submanifold  is a hypersurface of the contact manifold $(Y,\xi)$ that contains the Legendrian  and on which the contact form $\alpha$ restricts to a Liouville form. By the standard Legendrian neighborhood theorem there exist many symplectic ribbons that are isomorphic to the disk cotangent bundle of the Legendrian. Here we give an explicit description of the symplectic ribbon  $\mathcal{R}_L$ for the Legendrian double $L_{Leg}^d$.

\begin{lemma} \label{lemma_sympribbon}
    The symplectic ribbon $\mathcal{R}_L$ of the Legendrian double $L_{Leg}^d$ can be taken to be a subset of the pages outside of a neighborhood of the binding while near the binding it is consisting of the following embedding $$\tilde{\phi}: I\times U\to D^2\times \partial {P}$$
$$(t,p)\mapsto \left((t,\rho(t,z)),\varphi(\sigma(t,z),q)\right)$$ where $\varphi:(U,dz-\lambda)\to (\partial {P},\alpha)$ is the embedding of the standard neighborhood $U$ of $\Lambda$ which is given by a neighborhood of the zero section of the 1-jet space $J^1(\Lambda)$. Here $I\times U$ is identified with $ I\times I\times D^*\Lambda$ since $U\subset J^1(\Lambda)$ and the point $p=(z,q)\in I\times D^*\Lambda$. The maps $\rho$ and $\sigma$ have the following properties:
\begin{itemize}
    \item $\rho(t,z)=zh(t)$ for a non-negative bump function $h$ with the support on the closed interval $I$,
\item $\sigma(t,z)=zk(t)$ for the function $k$ which  takes the values

    $k(t)=\begin{cases}
     1& \text{for  } t<0,\\
      0& \text{for  } t=0,\\
      -1& \text{for  } t>0.
\end{cases}$
    \end{itemize}
\end{lemma}

\begin{proof}
    The standard Weinstein neighborhood of $L$ inside the page constitutes a symplectic ribbon except for the part contained near the binding, i.e. the  neighborhood of the  boundary $\Lambda$ of the Lagrangian submanifold $L$. The Legendrian submanifold $\Lambda$ has a standard neighborhood $U$ in the binding $\partial {P}$  which is given by a neighborhood of the zero section of the 1-jet space $J^1(\Lambda)$. The Weinstein neighborhood of $L$ can be assumed to have the following form near the binding (recall that the symplectization of the 1-jet space $J^1\Lambda$ is symplectomorphic to the cotangent bundle $T^*\R \times T^*\Lambda$, \cite{chantraine2025representations}) $$\phi: I\times U\to D^2\times \partial {P}$$
$$(t,p)\mapsto \left((t,0),\varphi(p)\right)$$
 where $\varphi:(U,dz-\lambda)\to (\partial {P},\alpha)$ is the embedding. The pull-back of the one-form of the neighborhood of the binding under the parametrization $\phi$ of the ribbon gives
$$\phi^*\left(e^{f(r)}\alpha+\dfrac{r^2}{2} d\theta\right)=e^{f(t)}(dz-\lambda).$$ Unfortunately, this is not a Liouville form, since the two-form given by its exterior differential is non-degenerate wherever $f'(t)=0$. We perturb the map $\phi$ as described in the statement of  Lemma \ref{lemma_sympribbon} to the map
$$\tilde{\phi}: I\times U\to D^2\times \partial {P}$$
$$(t,p)\mapsto \left((t,\rho(t,z)),\varphi(\sigma(t,z),q)\right).$$
 
The one-form defined on $D^2\times \partial {P}$ can be rewritten as $\dfrac{xdy-ydx}{2}+e^{f(r)}\alpha$. Pulling back the differential of this one-form along $\tilde{\phi}$ yields
    \vspace{1\baselineskip}
    
$\tilde{\phi}^\ast (dx\wedge dy+ e^{f(r)}d\alpha+ e^{f(r)}f^\prime(r)dr\wedge \alpha)$
$$=dt\wedge d(\rho(t,z))+ e^{f(t)}d(d(\sigma(t,z))-\lambda+e^{f(t)}f^\prime(t)dt\wedge (d(\sigma(t,z))-\lambda) $$
$$=(\rho_z(t,z)+ e^{f(t)}f^\prime(t)\sigma_z(t,z))dt\wedge dz-e^{f(t)}d\lambda-e^{f(t)}f^\prime(t)dt\wedge \lambda $$  which gives a positive symplectic two-form on the ribbon near binding, as sought.    
\end{proof}

After constructing a symplectic ribbon as in Lemma \ref{lemma_sympribbon}, one can  attach the twisted symplectic ribbon $\mathcal{R}_L$ to the page ${P}$ to yield a new Liouville hypersurface $P \cup \mathcal{R}_L$ of the contact manifold $(Y,\xi)$. In the case when $L$ is a disk, this construction re-produces the so-called \textit{stabilization of the page}; this is a standard construction that was described e.g.~in \cite{vanKoert2010}. We proceed to outline this construction from the point of view of abstract open books and contact manifolds. 

To stabilize an open book, we deform the page by attaching a Weinstein  $n$-handle to its $2n$-dimensional page along the Legendrian sphere 
in the binding that admits an exact Lagrangian disk filling. The contact manifold given by the new page with the monodromy consisting of the old monodromy composed with the Dehn twist around the sphere which is the new Lagrangian sphere in the new page. Note that the Lagrangian sphere is spanned by the Lagrangian disk and the core disk of the $n$-handle.  

In order to see that the new contact manifold defined by the deformed page with the deformed monodromy actually is contactomorphic to the original one, one can use canceling Weinstein surgeries in the following way.

Let $Y^{2n+1} = OB({P}, \eta, \psi)$ be a contact open book, and let $L$ be an $n$-dimensional Lagrangian disk  contained in the page $({P},\eta)$. To stabilize the open book, 
we begin by attaching a critical Weinstein $n$-handle to ${P}$ along $\partial L$, producing a 
new page ${P}_\Lambda$. The boundary $\Lambda$ corresponds to an isotropic sphere in the binding, so at the level of contact manifolds, stabilization starts 
with contact surgery along $\Lambda$.

From the viewpoint of symplectic cobordisms, we take a compact part $[0,1] \times Y$ and attach a subcritical Weinstein $n$-handle to $\{1\} \times Y$ along $\{1\} \times \Lambda$ using the framing of the symplectic normal bundle provided by the coordinates near the binding. Note that this produces a contact manifold that is topologically different from the one we started with. The next step in the construction of the abstract open book involves modifying the monodromy: in the new page ${P}_\Lambda$, we 
obtain a Lagrangian sphere $\tilde{L}$ by taking the union of the Lagrangian disk filling and the Lagrangian core disk obtained from the $n$-surgery. We may assume that $\tilde{L}$ is also a Legendrian sphere in $OB({P}_\Lambda, \tilde{\psi})$. The stabilized open book is then defined as
$$Y_\Lambda = OB({P}_\Lambda, \tilde{\eta}, \psi \circ \tau_{\tilde{L}})$$

\noindent where $\tau_{\tilde{L}}$ is the symplectic Dehn twist around $\tilde{L}$.

On the contact manifold level, this monodromy change corresponds to critical contact surgery along the Legendrian sphere $\tilde{L}$. This surgery, in turn, corresponds to an $n+1$-handle to the cobordism manifold along the Legendrian sphere $\tilde{L}$. The Legendrian sphere $\tilde{L}$ intersects the belt sphere of the previously attached $n$-handle at exactly one point resulting from the construction of stabilization. Thus, the stabilization procedure aligns perfectly with the handle cancellation and then we conclude that stabilization gives back the original contact manifold.

The stabilization procedure described above admits the following interpretation in terms of the page obtained by smoothing the union ${P}\cup \mathcal{R}_L$.

\begin{theorem}\label{prop_stab_ob}
    Let $L$ be a Lagrangian disk with Legendrian boundary in the page $(P,\eta)$ of an open book and denote by $\mathcal{R}_L$ the symplectic ribbon of its double. The smoothing of the union ${P}\cup \mathcal{R}_L $ that gives a Weinstein hypersurface obtained from $P$ by a Weinstein handle attachment along the Legendrian boundary of $L$ is a page of the open book obtained from the one above by the stabilization defined by the Lagrangian disk $L \subset P$. 
\end{theorem}
\begin{proof}
  Consider the exact pre-Lagrangian filling of half open book-type given by the Lagrangian disk $L$, which by Theorem \ref{thm:existence}  is the standard Lagrangian disk filling of the standard Legendrian sphere.
 
 We may thus assume that $P \cup L^d$ consists of an unknot union a Liouville hypersurface $P$ which is disjoint from the interior of its standard pre-Lagrangian disk filling. We then add a canceling Weinstein handles of index $n$ and $n+1$, where both attaching spheres are disjoint from $P \cup L^d$, but while the Legendrian attaching sphere of the $n+1$-handle passes through the pre-Lagrangian disk geometrically once transversely. See the left-hand side of Figure \ref{fig:handle slide} for a schematic picture in dimension 3. Using standard handle moves, we obtain the picture shown on the right, which corresponds to adding a subcritical handle attached along the boundary of $L$ followed by a critical handle attachment along the produced Legendrian. See Ding--Geiges \cite{ding2009handle} for the handle slide in dimension 3, and Casals--Murphy \cite{casals2019legendrian} for the higher dimensional version.
  A canceling handle pair is then introduced so as to pass through the Lagrangian disk filling. After performing a handle slide of the Legendrian sphere over the $n$-handle, the Legendrian sphere  becomes parallel to the critical handle, see Figure \ref{fig:handle slide}. Introducing a cancellation pair stabilizes the page of the open book, due to the discussion has been made above.
  \end{proof}

\noindent
\begin{figure}[htbp]
    \begin{minipage}{0.4\textwidth}
        \begin{tikzpicture}[scale=0.9, line width=1pt]
  \shade[ball color = gray!10, opacity = 0.3] (-2.7,0.2) circle (0.7);
\draw (-2.7,0.2) circle (0.7);
\draw[dashed] (-2.7,0.2) ellipse (0.7 and .2);
\shade[ball color = gray!10, opacity = 0.3] (2.7,0.2) circle (0.7);
\draw (2.7,0.2) circle (0.7);
\draw[dashed] (2.7,0.2) ellipse (0.7 and .2);
\draw (-2,0.2) -- (2,0.2);
\def\basecurve(#1,#2){
    plot[domain=-1:1, samples=200, smooth]
    (\x, {#1*(1-\x*\x)*sqrt(1-\x*\x)})
}
\foreach \y in {0.7} {
    \draw[red, line width=1pt]
       \basecurve(\y,0);
       \draw[blue, line width=2pt]
        plot[domain=-1:1, samples=200, smooth]
        (\x, {-\y*(1-\x*\x)*sqrt(1-\x*\x)});
}
\node[red] at (0.75,0.75) {$\Lambda_0$};
\node[blue] at (0.75,-0.75) {L};
\end{tikzpicture}
 \end{minipage}
 \hspace{2\baselineskip}
\begin{minipage}{0.4\textwidth}
    \begin{tikzpicture}[scale=0.9, line width=1pt]
         \shade[ball color = gray!10, opacity = 0.3] (-2.7,0.2) circle (0.7);
\draw (-2.7,0.2) circle (0.7);
\draw[dashed] (-2.7,0.2) ellipse (0.7 and .2);
\shade[ball color = gray!10, opacity = 0.3] (2.7,0.2) circle (0.7);
\draw (2.7,0.2) circle (0.7);
\draw[dashed] (2.7,0.2) ellipse (0.7 and .2);
\draw[red] (-2,0.2) -- (-1,0.2);
\draw[blue, line width=2pt] (-1,0.2)--(1,0.2);
\draw[red] (1,0.2)--(2,0.2);
\draw (-2,0.38) -- (2,0.38);
\node[red] at (1.5,-0.4) {$\Lambda'_0$};
\node[blue] at (0.1,-0.3) {L};
\draw[white] (-1,-0.7) -- (1,-0.7);
     \end{tikzpicture}
     \end{minipage}
    \caption{Left: A canceling pair consisting of an $n-1$ and $n$-surgery that links a standard Legendrian $n$-sphere $\Lambda_0$ (here $n=1$), by passing through the standard pre-Lagrangian filling of the standard sphere. Right: The effect of handle sliding the standard Legendrian sphere over the handle. The figure illustrates the three-dimensional case; however, an analogous handle slide exists in higher dimensions.}
    \label{fig:handle slide}
     \end{figure}
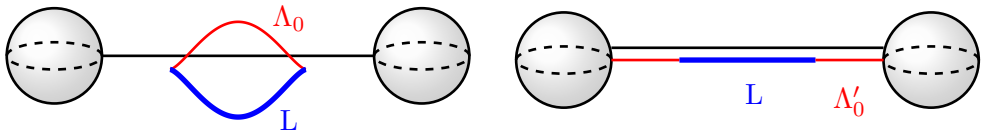
   \section{Legendrian suspension}
\label{suspension}
 One natural tool for constructing a wealth of Legendrians contained inside pages of open books was introduced in our previous work \cite{AY} and called \emph{Legendrian suspension}. This construction is used to produce a Legendrian sphere contained inside a $(p,q,2)$-Brieskorn variety that arises as a page of an open book decomposition. The construction is roughly made as follows. We start by taking \textit{suspension} of a Lefschetz fibration on $D^4$ with fibers Brieskorn variety of type $(p,q)$ and of a Legendrian $1$-knot which bounds a disk in $D^4$. The resulting Lagrangian is constructed in the Brieskorn variety of type $(p,q,2)$ by suspending the Lefschetz fibration and lifting the disk under a double branched cover. This construction can be generalized in order to produce a closed exact Lagrangians inside Seidel’s so-called suspended fiber of a Lefschetz fibration, \cite{Seidel}. For any Lefschetz fibration on a Liouville manifold $\hat{P}$ and a Legendrian submanifold in the boundary of ${P}$ which is sitting on a regular fiber of the Lefschetz fibration  and is fillable by a Lagrangian, we construct a Lagrangian submanifold in a regular fiber of the suspended Lefschetz fibration on $\hat{P}\times \mathbb{C}$ by lifting the Lagrangian filling under double branched cover. The construction takes any exact Lagrangian filling of a Legendrian and then suspends this filling to a closed exact Lagrangian contained inside a generic fiber of the suspended Lefschetz fibration. The result can thus be made into a Legendrian contained inside a page of an open book decomposition of the ideal contact boundary of $\hat{P}\times \C$.

We will now describe this construction in detail and with slightly more generality.
Let  $\hat{P}$ be a Liouville manifold which is given as the completion of the Liouville domain $({P},\eta)$  and $L$ be a Lagrangian submanifold  in $({P},\eta)$ which is cylindrical outside of a compact subset with  Legendrian boundary $\Lambda$. We now make the assumption that $L$ is strongly exact. If this is not the case, then one can simply modify the Liouville form $\eta$ in a compact subset.

Assume that $\pi: \hat{P}\to \mathbb{C}$ is a Lefschetz fibration on the Liouville manifold $\hat{P}$ such that the Lagrangian submanifold 
$\hat{L}$ is cylindrical outside of a compact subset and, moreover, projects to the line $(R-\epsilon,+\infty) \subset \C$ outside of the disk with radius $R-\epsilon$. On $\hat{P}$, we have the following coordinate:
$$t=\|{w}\circ \pi\|+\rho_F$$ 
where $t$ is the $\mathbb{R}_{>0}$ coordinate, $w$ is the  coordinate on the image $\mathbb{C}$ of the Lefschetz fibration $\pi$ and $\rho_F$ is defined on the fiber $\hat{F}$ of the Lefschetz fibration $\pi$. The map $\rho_F$ is defined to be  zero inside the fiber $\hat{F}$ and increasing in the cylindrical parts of the fibers. Outside of a compact subset of $\hat{P}$, the space $\hat{P}$ can be symplectically identified with $\hat{F} \times \C_w$, so that $\pi(x,w)=w$. Also, we can assume that $\Lambda$ is the subset of $\{\rho_F=0\}$ in the fiber. Outside of a compact subset $\hat{P}$ projects onto the set $\{\|w\|\geq \frac{R}{2}\}$.

Recall the definition of the \emph{suspension of the Lefschetz fibration $\pi$}, as defined by Seidel \cite{Seidel}, which is the Lefschetz fibration on the trivial product $\hat{P}\times \C$ given by $\pi^\sigma(p,z)=\pi(p)+z^2$ where $p\in \hat{P}$ and $z\in \mathbb{C}$. The generic fiber of $\pi^\sigma$ is called the \emph{suspension of $\hat{P}$}, and denoted by $\hat{P}^\sigma$. 

The double branched cover $\Pi:\hat{P}^\sigma\to \hat{P}$ along $\pi^{-1}(R)$ has two pre-image points for every point except on the branching locus given by the fiber $\pi^{-1}(R)$, where $\hat{P}^\sigma$ is the regular fiber $(\pi^\sigma)^{-1}(R)$ for $R\gg0$. If the Lagrangian filling $L$ has Legendrian boundary $\Lambda=\pi^{-1}(R) \cap \hat{L},$ and the cylindrical end of $\hat{L}$ projects to a line $(R-\epsilon,\infty)$ in $\C$ under the Lefschetz fibration $\pi$, then the two lifts of $L$ under the branched cover glue along the boundary to form a smooth embedded Lagrangian submanifold $L^\sigma \subset P^\sigma$. The Lagrangian property can be assumed to hold after a suitable deformation of the symplectic structure. Namely, since $\Pi$ is a holomorphic branched cover,  we may assume that it is an exact symplectomorphism outside of some neighborhood of the branching locus; thus $L^\sigma$ is Lagrangian there. Near the branching locus we may assume that the symplectic form is chosen so that the pre-image of $L$ is Lagrangian.

Combining the preceding constructions, we arrive at the following definition of the \emph{Legendrian suspension}.

\begin{definition}\label{definition:Leg_suspension}
    Let $\pi\colon \hat{P} \to \C$ be a Lefschetz fibration on the completion $\hat{P}$ of $(P,\eta)$, and $\hat{L}$ the completion of an exact filling $L \subset (P,\eta)$ with Legendrian boundary equal to $\pi^{-1}(R) \cap \hat{L}$, such that the projection of the cylindrical end of $\hat{L}$ under $\pi$ coincides with a line $(R-\epsilon, +\infty)$. Consider the induced open book on the ideal boundary $\partial_{\infty}(\hat{P}\times \C)$ corresponding to the suspended Lefschetz fibration $\pi^{\sigma}\colon \hat{P}\times \C \to \C$ with pages branched covers of $\pi^{-1}(R)$ for $R\gg0$. For a suitable deformation of the pages as in Lemma \ref{lemma_def_ob} the suspension $L^\sigma$, which is given by the pre-image of $L$ under the branched cover $P^\sigma \to P$, is strongly exact. In particular, $L^\sigma$ is an embedded Legendrian contained in a page, which we call the \emph{Legendrian suspension} of the Lagrangian filling $L$.
\end{definition}
The suspended fiber $\hat{P}^\sigma$ is obtained from $\hat{P}$ by attaching critical Weinstein handles along the boundaries of all Lefschetz thimbles of $\pi \colon \hat{P} \to \C$. Further, the Lefschetz fibration $\pi^\sigma $ is obtained from the Lefschetz fibration $\hat{P} \times \C$ by attaching critical Weinstein handles along the spheres formed by the Lefschetz thimbles union the latter Weinstein handles in $\hat{P}^\sigma$. See Seidel \cite{Seidel} for more details. This deformation corresponds to deforming the trivial open book induced by the trivial Lefschetz fibration on $\hat{P} \times \C$ by stabilization. More precisely:
\begin{proposition}\label{prop_stab_suspension}
    The  open book decomposition with page $P^\sigma$ induced by the Lefschetz fibration $\pi^\sigma$ is obtained from the trivial open book decomposition on $\partial_\infty(\hat{P}\times \C)$ with page $P$ by stabilizing each Legendrian sphere in $\partial P$ equal to a vanishing cycle of the Lefschetz fibration $\pi \colon \hat{P} \to \C$.
\end{proposition}

\begin{proof}
    We start with the trivial open book $({P}, \eta, id)$ as induced from the trivial Lefschetz fibration on $\hat{P}\times \mathbb{C}$. The suspension of the Lefschetz fibration produces the fiber $\hat{P}^\sigma=(\pi^\sigma)^{-1}(R)$.  Note that here the Lefschetz fibration $\pi^\sigma$ has as many critical values as of the Lefschetz fibration $\pi$. This leads us to the observation that the suspended fiber $\hat{P}^\sigma$ is obtained by attaching that many Lefschetz handles to the fiber $\hat{P}$ of the trivial Lefschetz fibration.
    By the discussion in the Section \ref{suspension}, the vanishing cycles of the Lefschetz fibration $\pi^\sigma$ arises as the double branched cover those of $\pi$. The vanishing cycles in $\hat{P}^\sigma$ are isotopic to the suspension of the vanishing cycles of the former Lefschetz fibration $\pi$. The page ${P}^\sigma$ of the induced open book is formed by attaching half-dimensional symplectic handles to the page ${P}$ and the monodromy is given by the positive Dehn twists about the suspended vanishing spheres. Hence, the new induced page is a multiple stabilization of the trivial Lefschetz fibration.   
\end{proof}

\subsection{Proof of Theorem \ref{theorem:leg_isotopy}}
\begin{proof}
The suspended Lefschetz fibration on the Liouville manifold $(\hat{P} \times \mathbb{C}, \lambda)$ is given by the map $$\pi^\sigma \colon \hat{P} \times \C \to \C, \qquad \pi^\sigma(p,z) := \pi(p) + z^2,$$
where $\pi \colon \hat{P} \to \C$ is the Lefschetz fibration on  $\hat{P}$, as described at the beginning of Section~\ref{suspension}. We deform the Lefschetz fibration $\pi
^\sigma$ through Lefschetz fibrations denoted by $ \pi^\sigma_t$ where $ \pi^\sigma_0= \pi^\sigma$ and $\pi^\sigma_t(p,z)=\pi_t(p)+z^2$. We proceed to define $\pi_t$.

Consider a smooth cut-off function $\rho:{\C}\to [0,1]$ satisfying
$$\rho(z)=\begin{cases}
     1& \text{if}\quad |z|\geq R-\epsilon,\\
      0& \text{if}\quad  |z|<R-2\epsilon.
\end{cases}$$
Using this cut-off function, define the perturbed suspension map
$${\pi}^\sigma_t(p,z)=(t\cdot\rho(\pi)+(1-t))\pi(p)+z^2.$$
This map can be readily seen to be a Lefschetz fibration when restricted to the pre-image of $\C \setminus D^2_{R-\epsilon}$ for each value $t \in [0,1]$ of the parameter. Moreover, the Lefschetz fibers are deformations of the original Lefschetz fibers over the same subset. There is an corresponding deformation of the induced compatible open book of the ideal contact boundary of $\hat{P} \times \C$.

There is again a branched cover ${(\pi^{\sigma}_t)^{-1}}(R) \to \hat{P}$. The suspension of $L^{\sigma}$ can be identified with the lift of $L$ under these branched covers. In addition, when $t=1$, this Legendrian can be seen to coincide with the Legendrian boundary at infinity of $\hat{L} \times \R$ for a suitable choice of contact-type hypersurface.

In the proof of Theorem~\ref{thm:flex_trivial_fib}, we show that the latter ideal boundary of \(\hat{L}\times\mathbb{R}\) inside  the ideal boundary $\partial_\infty(\hat{P}\times \C)$, as described in Example~\ref{ex:induced_ob}, coincides with the Legendrian double \(L^d_{Leg}\). Consequently, the corresponding Legendrian isotopy follows as desired.
  \end{proof}
   \nocite{*}
\bibliographystyle{alpha}
\bibliography{references}
\bigskip
\end{document}